\documentclass[11pt]{article}

\usepackage{graphicx, amsmath, amsthm, amsfonts, enumerate, authblk}
\usepackage{color}

\newcommand\ep{\epsilon}

\newcommand{\curl}[2]{\textrm{curl}_{#1}\hspace{0.05cm} #2}
\newcommand{\ws}{\mathbb{R}^3}
\newcommand{\norm}[2]{\Vert #1 \Vert_{#2}}
\renewcommand{\div}[2]{\textrm{div}_{#1}\hspace{0.05cm} #2}

\newcommand{\RR}{\mathbb{R}}

\newtheorem{thm}{\bf Theorem}[section]
\newtheorem{prop}{\bf Proposition}[section]
\newtheorem{lem}{\bf Lemma}[section]

\newtheorem{remark}{\bf Remark}[section]
\newtheorem{corollary}{\bf Corollary}[section]

\numberwithin{equation}{section}

\makeatletter
\def\blfootnote{\xdef\@thefnmark{}\@footnotetext}
\makeatother

\begin{document}

\title{\sc Asymptotic behaviour of the spectra of systems of Maxwell equations in periodic composite media with high contrast}
\author[1]{Kirill Cherednichenko}
\author[1]{Shane Cooper}
\affil[1]{Department of Mathematical Sciences, University of Bath, Claverton Down, Bath, BA2 7AY, United Kingdom}






\maketitle


\begin{abstract}
	We analyse the behaviour of the spectrum of the system of Maxwell equations of electromagnetism, with rapidly oscillating periodic coefficients, subject to periodic boundary conditions on a ``macroscopic'' domain $(0,T)^d, T>0.$ We consider the case when the contrast between the values of the coefficients in different parts of their periodicity cell increases as the period of oscillations $\eta$ goes to zero. We show that the limit of the spectrum as $\eta\to0$ contains the spectrum of a ``homogenised'' system of equations that is solved by the limits of sequences of eigenfunctions of the original problem. We investigate the behaviour of this system and demonstrate phenomena not present in the scalar theory for polarised waves.
\end{abstract}


\section{Introduction}

The behaviour of systems (of Maxwell equations) with periodic coefficients in the regime of ``high contrast'',  or ``large coupling'' {\it i.e.} when the ratio between material properties of some of the constituents within the composite is large, is understood to be of special interest in applications. This is due to the improved band-gap properties of the spectra for such materials compared to the usual moderate-contrast composites. A series of recent studies have analysed asymptotic limits
of scalar high-contrast problems, either in the strong $L^2$-sense  (see \cite{Zhikov2000}, \cite{Zhikov2004}) or in 
the norm-resolvent  $L^2$-sense,  see \cite{CC_ARMA}.  These have resulted in sharp operator convergence estimates in the homogenisation of such problems ({\it i.e.} in the limit as the period tends to zero) and have provided a link between the study of effective properties of periodic media and the behaviour of waves in such media, in particular their scattering characteristics. This suggests a potential for applications of the abstract operator theory to the study of such problems.  The studies  have also highlighted the need  to extend  the classical compactness techniques in homogenisation to cases when the symbol of the operator involved is no longer uniformly positive definite, thus leading to ``degenerate'' problems. The work \cite{KamSmysh2013} has opened a way to one such extension procedure, based on a  ``generalised Weyl decomposition", from the perspective of the strong $L^2$-convergence.  

 The set of tools developed in the literature is now poised for the treatment of vector problems with degeneracies such as the linearised elasticity equations and the Maxwell equations; these examples are typically invoked in the physics and applications literature, and are prototypes for wider varieties of partial differential equations (PDE).  The recent work \cite{ZhPast2013} has studied the spectral behaviour of periodic operators with rapidly oscillating coefficients in the context of linearised elasticity. It shows that the related spectrum exhibits the phenomenon of ``partial'' wave propagation, depending on the number of eigenmodes available at each give frequency. This is close in spirit to the work of \cite{SmyshlyaevPartial}, where ``partial wave propagation'' was studied for a wider class of vector problems, with a general high-contrast anisotropy. 

The high-contrast system of Maxwell equations poses an analytic challenge in view of the special structure of the 
``space of microscopic oscillations" (using the terminology of  \cite{KamSmysh2013}), which consists of the functions that are 
curl-free on the ``stiff'' component, in the case of a two-component composite of a ``stiff'' matrix and ``soft'' inclusions. In the work \cite{maxCh} we analysed the two-scale structure of solutions to the high-contrast system of Maxwell equations in the low-frequency limit,  and derived the corresponding system of homogenised equations, by developing an appropriate compactness argument on the basis of the general theory of  \cite{KamSmysh2013}. In the present paper we consider the associated wave propagation problem for monochromatic waves of a given frequency  by constructing two-scale asymptotic series for eigenfunctions. We justify these asymptotic series by demonstrating that for each element of the spectrum of the homogenised equations their exist convergent eigenvalues and eigenfunctions for the original heterogeneous problem. Our analysis is set in the context of a ``supercell'' spectral problem, {\it i.e.} the problem of vibrations of a square-shaped domain with periodicity conditions on the boundary (equivalently seen as a torus).
The problem of the ``spectral completeness'' of the homogenised description in question remains open: it is not known, for the full-space problem, whether there may exist sequences of eigenvalues converging to a point outside the spectrum of the homogenised problem. We shall address this in a future publication, using the method we developed in \cite{CC_ARMA}.

\section{Problem formulation and main results}

In this paper we consider Maxwell equations for a  three-dimensional two-component periodic dielectric composite when the dielectric properties of the constituent materials exhibit a high degree of contrast between each other. 
We assume that the reference cell $Q:=[0,1)^3$ contains an inclusion $Q_0,$ which is an open set with sufficiently smooth boundary.
We also assume that the ``matrix'' $Q_1 : = Q \backslash \overline{Q_0}$ is simply connected Lipschitz set. 

We consider a composite with high contrast in the dielectric permittivity $\epsilon_\eta=\epsilon(x/\eta)$ at points 
$x\in\eta(Q_1+m),$ $m\in{\mathbb Z}^3,$ and $x\in\eta(Q_0+m),$ $m\in{\mathbb Z}^3,$ namely
$$
\ep_\eta(y) = \left\{ \begin{array}{cr} \eta^{-2}\epsilon_0(y), & y \in Q_0,\\[0.3em]

\epsilon_1(y),  & y \in Q_1,   \end{array} \right.
$$
where $\eta\in(0,1)$ is the period and $\epsilon_0,$ $\epsilon_1$ are continuously differentiable $Q$-periodic positive-definite scalar functions. 

We also assume moderate contrast in the magnetic permeability, and for simplicity of exposition we shall set $\mu \equiv 1$.  We consider the open cube ${\mathbb T}:=(0,T)^3$ and those values of the parameter $\eta$ for which $T/\eta\in{\mathbb N}.$ By re-scaling the spatial variable (which can also be viewed as  non-dimensionalisation) we assume that $T=1$ and that $\eta^{-1}\in{\mathbb N}.$   We shall study the behaviour of the magnetic component $H^\eta$ of the electromagnetic wave of frequency $\omega$ propagating through the domain $\mathbb T$ occupied by a dielectric material with permittivity $\ep_\eta(x/\eta).$ More precisely, 
we consider pairs $\bigl(\omega_\eta, H^\eta\bigr) \in \mathbb{R}_+ \times [H^1_\#(\mathbb T)]^3$ satisfying the system of equations
\begin{equation}
\label{e1.1}
\curl{}{\bigl( \ep_\eta^{-1}\bigl(\tfrac{x}{\eta}\bigr) \curl{}{H^\eta} \bigr)} = \omega_\eta^2 H^\eta.
\end{equation} 
Notice that solutions of (\ref{e1.1}) are automatically solenoidal, {\it i.e.} ${\rm div} H^\eta=0.$

We seek solutions to the above problem in the form of an asymptotic expansion
\begin{equation}
\label{e1.2}
H^\eta(x) = H^0 \bigl(x,\tfrac{x}{\eta}\bigr) + \eta H^1 \bigl(x , \tfrac{x}{\eta} \bigr) + \eta^2 H^2 \bigl(x,\tfrac{x}{\eta} \bigr)+ ...,
\end{equation}
where the vector functions $H^{j}(x,y),$ $j=0, 1, 2,...,$ are $Q$-periodic in the variable $y.$ Substituting \eqref{e1.2} into \eqref{e1.1} and gathering the coefficients for each power of the parameter $\eta$ results in a system of recurrence relations for $H^j,$ $j=0, 1, 2,...,$ see Section \ref{expansionsection}. In particular, the function $H^0$ is an eigenfunction of a limit (``homogenised'') system of PDE, as described in the following theorem.
\begin{thm}
	\label{maintheorem}
	Consider the constant matrix
	\[
	A^{\rm hom} : = \int_{Q_1}\epsilon_1^{-1}(y)\big({\rm curl}N(y) + I \big) \,\mathrm{d}y,
	\]
	where the vector-function $N$ is a solution to the ``unit-cell problem''
	\begin{equation}
	{\rm curl}\Bigl(\epsilon_1^{-1}\bigl[{\rm curl}N+I\bigr]\Bigr)=0\ \ \ {\rm in}\ Q_1, \ \ \ \epsilon_1^{-1}\bigl({\rm curl}N+I\bigr)\times n=0\ \ {\rm on}\ {\partial Q_0},\ \ \ N\ {\rm is}\ \text{$Q$-{\rm periodic}},
	\label{ELorig}
	\end{equation}
	where $n$ is the exterior normal to $\partial Q_0.$
	
	Suppose that $\omega\in{\mathbb R}_+$ and $H^0(x,y) = u(x) + \nabla_y v(x,y) + z(x,y),$ where the triplet\footnote{For a cube  ${\mathbb T},$ we denote by $H^1_\#(\mathbb T),$ $H^1_{\rm \# curl}(\mathbb T),$ the 
		closures of the set of $\mathbb{T}$-periodic smooth functions with respect to the norm of $H^1(\mathbb T)$ and the norm 
		\[
		\biggl(\int_{\mathbb T}\vert\cdot\vert^2+\int_{\mathbb T}\vert{\rm curl}\cdot\vert^2\biggr)^{1/2},
		\] 
		respectively.}
	$(u,v,z) \in \bigl[H^1_{\rm \# curl}(\mathbb T)\bigr]^3 \times L^2\bigl(\ws ; H^1_\#(Q)\bigr) \times\bigl[L^2\bigl(\mathbb T; H^1_0(Q_0)\bigr)\bigr]^3,$ satisfies the system of equations
	\begin{alignat}{3}
	{\rm curl}_{x}\bigl(A^{\rm hom}{\rm curl}_{x} u(x)\bigr)  = \omega^2 \Bigl(u(x) + \int_{Q_0}z (x,y) dy \Bigr) , &  \qquad& x\in{\mathbb T},\label{me1}  \\[0.2em]
	{\rm div}_{y}\big( \nabla_y v(x,y) + z(x,y) \big)  = 0 \label{me2}, &\qquad & (x,y)\in{\mathbb T}\times Q, \\[0.6em]
	{\rm curl}_{y}\bigl(\epsilon_0^{-1}(y){\rm curl}_{y}z(x,y)\bigr) = \omega^2 \big( u(x) + \nabla_y v(x,y) + z(x,y) \big), &\qquad & 
	(x,y)\in {\mathbb T}\times Q_0.\label{me3}
	\end{alignat}
	
	Then: 
	
	1) There exists at least one eigenfrequency $\omega_\eta$ 
	for (\ref{e1.1}) such that 
	$\label{omegavicinity}
	\vert \omega_\eta-\omega\vert<C\eta,$
	with an $\eta$-independent constant $C>0.$
	
	2) Consider the finite-dimensional vector space 
	\[
	X_\eta:={\rm span}\bigl\{H^\eta: (\ref{e1.1}){\rm \ holds,\ where\ } \omega_\eta {\rm \ satisfies}\ (\ref{omegavicinity})\bigr\}.
	\]
	There exists  an  $\eta$-independent constant $\widehat{C}>0$ such that 
	${\rm dist}\bigl(H^0, X_\eta\bigr)<\widehat{C}\eta$.

\end{thm}

	\label{ahomrem}
	The matrix $A^{\text{\rm hom}}$ is described by solutions to certain degenerate ``cell problems", as follows.
	Consider the spaces
	\begin{equation}
\label{spaceV}
V : = \Big\{v \in [H^1_{\#}(Q)]^3 : \text{$\curl{}\, {v =0}$ in $Q_1$} \Big\}
\end{equation}
	and $V^\perp,$ the orthogonal complement of $V$ in $[H^1_\#(Q)]^3$ with respect to the equivalent $H^1$-norm
\begin{equation*}
\norm{v}{H} : = \biggl(\left\vert\int_Q v \right\vert^2 + \int_Q \vert \nabla v \vert^2\biggr)^{1/2},
\end{equation*}
associated with the inner product
\begin{equation*}
(v,w)_H : =  \biggl(\int_Q v\biggr)\cdot\biggl(\int_Q w\biggr) + \int_Q  \nabla v \cdot \nabla w.
\end{equation*}

Then 
\begin{equation}
\label{gamma}
A^{\rm hom}\xi = \int_{Q} \epsilon_1^{-1} \left( {\rm curl}{N}_\xi + \xi \right),\ \ \ \ \xi\in{\mathbb R}^3,
\end{equation}
where $N_\xi,$ $\xi\in{\mathbb R}^3,$
is the unique (weak) solution in $V^\perp$ to the problem  (\ref{ELorig}), i.e.
\begin{equation}
	\label{ahom1.1}
	\int_Q \epsilon^{-1}_1 \left( \curl\,N_\xi + \xi\right)  \cdot \curl\,{\varphi} = 0, \qquad \forall\varphi \in V^\perp.
\end{equation}
Existence and uniqueness of $N_\xi$ is discussed in Section \ref{expansionsection}.

Notice that
\begin{equation}
A^{\rm hom}\xi\cdot\xi=\min_{U\in[H^1_{\#}(Q)]^3}\int_{Q_1}\epsilon_1^{-1}({\rm curl}\,U+\xi)\cdot({\rm curl}\,U+\xi),\ \ \ \xi\in{\mathbb R}^3.
\label{newrep}
\end{equation}
Indeed, 
for the functional
$$
F_\xi(U) : = \int_{Q_1}\epsilon_1^{-1}({\rm curl}\,U+\xi)\cdot({\rm curl}\,U+\xi),
$$ 
we find $F_\xi(U) = F_\xi(P_{V^\perp}U)$ for all $U\in [H^1_{\#}(Q)]^3$, where $P_{V^\perp}$ is the orthogonal projection onto $V^\perp$. Therefore, without loss of generality, $F_\xi$ can be minimised on $V^{\perp}$ for which \eqref{ahom1.1} is the corresponding Euler-Lagrange equation.

The variational formulation (\ref{newrep}) allows one to obtain a representation for 
		the matrix  $ \epsilon^{\rm hom}_{\rm stiff}$ such that
	\begin{gather}
	\epsilon^{\rm hom}_{\rm stiff} \xi \cdot \xi : = \inf_{ \substack{ u \in H^1_{\#}(Q), \\ \nabla u  = -\xi \,{\rm in}\,Q_0}} \int_{Q_1}\epsilon_1\left(  \nabla u + \xi \right) \cdot \left( \nabla u + \xi \right),\ \ \ \ \xi\in{\mathbb R}^3, \label{stiffmatrix}
	\end{gather}
	which arises in the homogenisation of periodic problems with stiff inclusions, see \cite[Section 3.2]{JKO}.\footnote{The Euler-Lagrange equation for  \eqref{stiffmatrix} is as follows:  find $u$ such that $\nabla u = - \xi$ in $Q_0$ and 
	$$
	\int_{Q_1}\epsilon_1\left( \nabla u + \xi \right) \cdot \nabla \phi = 0\ \ \ \ \ \ \forall\phi \in H^1_{\#}(Q),\ \ \ 
	\nabla \phi = 0\ {\rm in}\ Q_0.$$ 
	The equivalent ``strong'' form of the same problem is to find a $Q$-periodic function $u$ such that
	\begin{equation*}
	{\rm div}\left(\epsilon_1\left( \nabla u + \xi \right)\right)= 0 \ \  {\rm in}\ Q_1,\ \ \ \
	\int_{\partial{Q_0}}\epsilon_1\left( \nabla u + \xi \right) \cdot n  = 0,\ \ \  u\ {\rm is\ continuous\ across\ }\partial Q_0,\ \ \   
	\nabla u  = - \xi,\ \ {\rm in}\ Q_0. 
	\end{equation*}
}

		Indeed, as shown in \cite[p.\,101]{JKO}, 
	the following representation holds:
\begin{equation}
	\bigl( \epsilon^{\rm hom}_{\rm stiff} \bigr)^{-1}\xi\cdot\xi = \inf_{ \substack{ v \in [L^2(Q)]^3_{\rm sol},  \\ \langle v\rangle = 0}} \ \ \int_{Q_1}\epsilon_1^{-1}\left(v+\xi\right) \cdot \left(v+\xi\right),\ \ \ \xi\in{\mathbb R}^3.
	\label{v_rep}
	\end{equation}
Notice that for each vector $v$ in (\ref{v_rep}) there exists $U_v\in[H^1_{\#}(Q)]^3$ such that $v={\rm curl}\,U_v,$ see \cite[pp. 6--7]{JKO}, and hence  
\[
\int_{Q_1}\epsilon_1^{-1}\left(v+\xi\right) \cdot \left(v+\xi\right)=\int_{Q_1}\epsilon_1^{-1}\bigl({\rm curl}(P_{V^\perp}U_v)+\xi\bigr) \cdot \bigl({\rm curl}(P_{V^\perp}U_v)+\xi\bigr).
\]
It follows that for all $\xi\in{\mathbb R}^3$ one has
\[
\bigl( \epsilon^{\rm hom}_{\rm stiff} \bigr)^{-1}\xi\cdot\xi = \inf_{U\in[H^1_{\#}(Q)]^3}\int_{Q_1}\epsilon_1^{-1}\bigl({\rm curl}(P_{V^\perp}U)+\xi\bigr)\cdot 
\bigl({\rm curl(}P_{V^\perp}U)+\xi\bigr)=A^{\rm hom}\xi\cdot\xi.
\]
\section{On the spectrum of the limit problem}
In this section we study the set of values $\omega^2$ such that there exists a non-trivial triple $(u,v,z)$ solving the two-scale limit spectral problem \eqref{me1}--\eqref{me3}.
\subsection{Equivalent formulation and spectral decomposition of the limit problem}
Let $G$ be the Green function for the scalar periodic Laplacian, {\it i.e.} for all $y \in Q$ one has
\begin{equation*}
-\Delta G(y)=\delta_{0}(y)-1, \qquad y \in Q, \\
\hspace{1.3cm}\text{$G$ is $Q$-periodic,}
\end{equation*}
where $\delta_0$ is the Dirac delta-function supported at zero, on $Q$ considered as a torus. Then, as the functions $v,$ $z$ solve \eqref{me2}, we have 
$v(x,\cdot) = {G*({\rm{div_y}}_{}z)}(x,\cdot),$ and \eqref{me3} takes the form
\begin{equation}
\curl{y}\bigl(\epsilon_0^{-1}(y){\curl{y}{z(x,y)}}\bigr)= \omega^2 \left( u(x) +  \nabla_{y}\int_{Q_0} { G}(y-y')\,{\rm div}_{y'} z(x,y')\,dy' + z(x,y) \right), \ \ \ \ (x,y)\in{\mathbb T}\times Q_0. 
\label{Peqn_0}
\end{equation}
For the case $\omega=0$ the set of solutions $z$ to (\ref{Peqn_0}) subject to the condition $z(x, y)=0,$ $x\in{\mathbb T},$ 
$y\in\partial Q_0,$ is clearly given by $L^2({\mathbb T}, {\mathcal H}_0),$ where ${\mathcal H}_0:=\{u\in[H^1_0(Q_0)]^3: {\rm curl}\,u=0\}.$ 

Further, for $\omega\neq 0,$ as (\ref{Peqn_0}) is linear in $u(x)$ and ${\rm curl}_y\nabla_y=0$, we set 
\begin{equation}
  \nabla_{y}\int_{Q_0}{ G}(y-y')\,{\rm div}_{y'} z(x,y')\,dy'+z(x,y)=\omega^2B(y)u(x),
\label{decompose}
\end{equation}
where $B$ is a $3\times3$ matrix function whose column vectors $B^{j}$, $j=1,2,3,$ are solutions in $[H^1_{\#}(Q)]^3$ to the system
\begin{eqnarray}
& &\curl{}\bigl(\epsilon_0^{-1}{\curl{}{B^j}}\bigr) =e_j+ { \omega^2 B^j} \ \ \ \  {\rm in}\ {Q_0},\label{Peqn}\\[0.4em]
& &
{\rm curl}\,B^j=0,\ \ {\rm in}\ Q_1,\label{cond1}\\[0.4em]
& &{\rm div}\,B^j=0,\ \ {\rm in } \ Q,\label{cond2}\\[0.4em]
& &a(B^j)=0,\label{cond3}
\end{eqnarray} 
where $e_j,$ $j=1,2,3,$ are the Euclidean basis vectors and $a(B^j)$ is the ``circulation'' of $B^j,$ { that is defined as the continuous extension, in the sense of the $H^1$ norm, of the map given by $a(\phi)_i = \int_0^1 \phi_i ( t e_i) dt$, $i=1,2,3$, for $\phi \in [C^\infty(Q)]^3$}. Note that, since $B^j \in [H^1(Q)]^3$, the equation \eqref{cond1} implies 
$\epsilon_0^{-1}\curl{}{B^j}\times n \vert_{-} =0$ on $\partial Q_0$. Furthermore,  the system (\ref{Peqn})--\eqref{cond3} implies  the variational problem of finding $B^j \in [H^1_{\#}(Q)]^3,$ subject to the constraints \eqref{cond1}--\eqref{cond3}, such that the following identity holds:
\begin{equation}
\label{weakB}
\int_{Q_0}\epsilon_0^{-1}{\rm curl}\,B^j\cdot{\rm curl}\,\varphi=\int_{Q}e_j\cdot\varphi+{\omega^2}\int_{Q}B^j\cdot\varphi \hspace{.5cm} \text{ $\forall \varphi\in[H^1_{\#}(Q)]^3$ satisfying (\ref{cond1})--(\ref{cond3}).}
\end{equation}
Indeed, functions $\varphi\in[H^1_{\#}(Q)]^3$ which satisfy (\ref{cond1}),(\ref{cond3}) admit (see Lemma \ref{lem:spaceV} below) the representation $\varphi = \nabla p + \psi$, $p \in H^2_{\#}(Q)$, $\psi \in [H^1_0(Q_0)]^3$. Therefore, it is straightforward to show  \eqref{weakB} holds for $\varphi = \psi$ if and only if \eqref{Peqn} holds. Similarly,  one can show \eqref{weakB} holds for $\varphi = \nabla p$ if and only if \eqref{cond2} holds.

Substituting the representation (\ref{decompose}) 
into \eqref{me1} and using the fact that 
	$$
	\int_Q \big( \nabla_{y'} G*({\rm{div_{y}}}_{}z)(x,y') + z(x, y') \big) dy' = \int_{Q_0}  z(x,y)  dy,
	$$
leads to the operator-pencil spectral problem
\begin{equation}
\label{me7}
\curl{}{\bigl(A^{\text{hom}}\curl{} u(x)\bigr)}  =  \Gamma(\omega) u(x),  \qquad x \in {\mathbb T}, 
\end{equation}
where $\Gamma$  is a {matrix-valued} function that vanishes at $\omega=0,$ and for $\omega\neq0$ has elements
\begin{equation}
\label{gamma}
\Gamma_{ij}(\omega)= \omega^2\left(\delta_{ij} +\omega^2  { \int_Q B_{i}^j }\right),\ \ \  i,j=1,2,3.
\end{equation}
We denote by ${\mathcal H}_1$ the space of vector fields in $[H^1_{\#}(Q)]^3$ that satisfy the conditions (\ref{cond1})--(\ref{cond3}). It can be shown\footnote{Note that $|||\cdot|||:=\bigl(\int_{Q_0}\epsilon_0^{-1}\vert{\rm curl}\cdot\vert^2\bigr)^{1/2}$ is a norm in 
${\mathcal H}_1$ equivalent to the {$[H^1(Q)]^3$}-norm, due to the fact that 
$\bigl(\vert a(\cdot)\vert^2+\Vert{\rm div}\,\cdot\Vert^2_{L^2(Q)}+\Vert{\rm curl}\,\cdot\Vert^2_{[L^2(Q)]^3}\bigr)^{1/2}$ is an equivalent norm in the space $u\in[H^1_{\#}(Q)]^3$. Therefore, 
the equation $\curl{}\ep^{-1}_0\curl{}u=\lambda u,$ $u\in{\mathcal H}_1,$ can be written as $\lambda^{-1}u=Ku$ in the sense of the ``energy" inner product generated by the norm $|||\cdot|||$ and $K$ is a compact self-adjoint operator in $({\mathcal H}_1, |||\cdot|||).$ The claim then follows by a standard Hilbert-Schmidt argument.
}
that there exist countably many pairs $(\alpha_k, r_k) \in \mathbb{R}\times 
{\mathcal H}_1$  such that $\norm{r_k}{[L^2(Q)]^3}=1$ and
\begin{equation*}
\curl{}(\epsilon_0^{-1}{\curl{}{r_k}})= \alpha_k
r_k 
\ \ \ \ { {\rm in}\ Q_0.} 
\end{equation*}
Moreover, the sequence $(r^k)_{k\in\mathbb{N}}$ can be chosen to form an orthonormal basis of the {closure $\overline{\mathcal H}_1$ of ${\mathcal H}_1$ in $[L^2(Q)]^3$ and, upon a} suitable rearrangement, one has
\[
0<\alpha_1\le\alpha_2\le ...\le\alpha_k\le ...\stackrel{k\to\infty}\longrightarrow\infty.
\]
Performing a decomposition\footnote{When applying the standard Fourier representation approach with respect to the basis 
$(r^k)_{k\in{\mathbb N}},$ the vector $e_j$ in the right-hand side of (\ref{Peqn}) is treated as an element of the 
``dual'' of $\overline{\mathcal H}_1,$ the space of linear continuous functionals on ${\mathcal H}_1.$}
of the functions $B^j,$ $j=1,2,3,$ with respect to the above basis yields
$$
B_i^j =\sum_{k=1}^\infty \frac{\int_Q r^k_j}{\alpha_k - \omega^2}r^k_i,\ \ \ \ \ \ \ \ \ \ \omega^2\notin\cup\{\alpha_k\}_{k=1}^\infty,
$$
where $r^k_j,$ $j=1,2,3,$ are the components of the vector $r^k,$ $k\in{\mathbb N}.$  

Consider the functions $\phi^k\in[H_0^1(Q_0)]^3,$ $k\in{\mathbb N},$ that solve the non-local problems
\begin{equation}
\curl{}\bigl(\epsilon_0^{-1}(y){\curl{}{\phi^k}}(y)\bigr)= \alpha_k\biggl(\nabla\int_{Q_0}{ G}(y-y')\,{\rm div}\,\phi^k(y')\,dy'+\phi^k(y)\biggr),\ \ \ \ y\in Q_0,
\label{nonlocal_eq}
\end{equation}
and satisfy the orthonormality conditions
\[
\int_{Q_0}\int_{Q_0}\bigl(\nabla^2{ { G}}(y-y')+I\bigr)\phi_j(y)\cdot\overline{\phi^k(y')}\,dy\,dy'=\delta_{jk},\ \ \ j,k=1, 2,,...,
\]
where $\nabla^2{ {G}}$ is the Hessian matrix of ${{ G}}.$
Using the formula
\[
r^k(y)=\nabla\int_{Q_0}{ G}(y-y')\,{\rm div}\,\phi^k(y')\,dy'+\phi^k(y),\ \ \ \ \ y\in Q, 
\]
we obtain the following representation for $\Gamma:$ 
 \begin{equation}
 \label{me8}
 \Gamma_{ij}(\omega)= \omega^2\delta_{ij} + \omega^4\sum_{k=1}^\infty \frac{\left( \int_{Q_0} \phi^k_i \right) \left( \int_{Q_0} \phi^k_j \right)}{\alpha_k - \omega^2},\ \ \ \ i,j=1,2,3,\ \ \ \ \ \ \ \ \ \omega^2\notin\{0\}\cup\{\alpha_k\}_{k=1}^\infty.
 \end{equation}


\subsection{{Analysis of the limit spectrum}}
Now, we consider the Fourier expansion for the function $u$ in \eqref{me7}:
\[
u(x)=\sum_{m\in{\mathbb Z}^3}\exp(2\pi{\rm i} m\cdot x)\hat{u}(m),\ \ \ \ \ \ \hat{u}(m):=\int_{\mathbb T}\exp(-2\pi{\rm i}m\cdot x)u(x)\,dx,
\]
where the integral is taken component-wise. {As $u$} solves \eqref{me7}, the coefficients $\hat{u}(m)$ 
satisfy the equation
\begin{equation}
\mathcal{M}(m) \hat{u}(m) = \Gamma(\omega) \hat{u}(m),\ \ \ \ m\in{\mathbb Z}^3,
\label{limfour}
\end{equation}
with the {matrix-valued} function ${\mathcal M}$ {is} given by
\begin{equation*}
\mathcal{M}_{lp}(m) = 4\pi^2\varepsilon_{ils}m_s A^{\rm hom}_{ij}\varepsilon_{jpt}m_t=4\pi^2(e_l\times m)\cdot A^{\rm hom}(e_p\times m),\ \ \ m\in{\mathbb Z}^3,\ \  l,p=1,2,3,
\end{equation*}
where $e_j,$ $j=1,2,3$ are the Euclidean basis vectors.  Here $\varepsilon$ is the Levi-Civita symbol:
\[
\varepsilon_{jkl}=\left\{\begin{array}{rc}1,&\ \ (jkl)=(123), (231), (312),\\[0.3em]-1,&\ \ (jkl)=(132), (321), (213),\\[0.3em] 0,&\ \ \ \ {\rm otherwise}.\ \ \ \ \ \ \ \ \ \ \ \ \ \ \ \ \ \ \ \ \ \ \ \end{array}\right.
\]
Notice that, for all $m\in{\mathbb Z}^3\setminus\{0\},$ zero is a simple eigenvalue of $\mathcal{M}(m)$ with eigenvector $m,$ and since 
the matrix $A^{\rm hom}$ is symmetric and positive-definite, the values of ${\mathcal M}$ are also symmetric and positive-definite on vectors $\xi$ such that $\xi \cdot m =0.$ In particular, for all $m \in \mathbb{Z}^3,$  one has
\begin{gather}
\Gamma(\omega) \hat{u}(m) \cdot m = 0
\label{solv_condition}
\end{gather}
whenever $\hat{u}(m)$ is a solution to (\ref{limfour}). 
Denote $\tilde{m}:=\vert m\vert^{-1}m$ and notice that ${\mathcal M}(m)=\vert m\vert^2{\mathcal M}(\tilde{m})$.  
Further, we denote by $\tilde{e}_1(\tilde{m})=\bigl(\tilde{e}_{11}(\tilde{m}), \tilde{e}_{12}(\tilde{m}), \tilde{e}_{13}(\tilde{m})\bigr)$ and $\tilde{e}_2(\tilde{m})=\bigl(\tilde{e}_{21}(\tilde{m}), \tilde{e}_{22}(\tilde{m}), \tilde{e}_{23}(\tilde{m})\bigr)$ the normalised  eigenvectors of 
the matrix $\mathcal{M}(\tilde{m})$ corresponding to its two positive eigenvalues $\lambda_1(\tilde{m})$  and $\lambda_2(\tilde{m})$ {respectively.}

We write $\hat{u}(m)$ in terms of the basis  $\big( \tilde{e}_1(\tilde{m}),\tilde{e}_2(\tilde{m}),\tilde{m}\big),$ as follows:  
\begin{equation*}
\hat{u}(m)=C(\tilde{m})^\top\tilde{u}(\tilde{m}) + \alpha(\tilde{m})\tilde{m},\  \ \ \ \ \tilde{u}(\tilde{m}) \in \RR^2, \alpha(\tilde{m})\in \mathbb{R},\ \  
C(\tilde{m})=  
\left( \begin{matrix}
\tilde{e}_{11}(\tilde{m}) & \tilde{e}_{12}(\tilde{m}) & \tilde{e}_{13}(\tilde{m})\\[0.3em]\tilde{e}_{21}(\tilde{m}) & \tilde{e}_{22}(\tilde{m}) & \tilde{e}_{23}(\tilde{m})
\end{matrix} \right).
\end{equation*}
 { Finding a} non-trivial solution to the problem (\ref{limfour}), \eqref{solv_condition} is equivalent to determining 
 $\bigl(\tilde{u}(\tilde{m}),\alpha(\tilde{m})\bigr) \in \mathbb{R}^3\setminus\{0\}$ such that
\begin{equation}
\label{lim:fourfinal}
\begin{aligned}
&\vert m \vert^2 \Lambda(\tilde{m}) \tilde{u}(\tilde{m})  =  C(\tilde{m}) \Gamma(\omega)C(\tilde{m})^\top\tilde{u}(\tilde{m}) +
\alpha(\tilde{m})C(\tilde{m})  \Gamma(\omega) \tilde{m},\ \ \ \  
\\[2.6pt]
& \Gamma(\omega) C(\tilde{m})^\top\tilde{u}(\tilde{m})\cdot \tilde{m}  = -  \alpha(\tilde{m})\Gamma(\omega) \tilde m \cdot \tilde m,
\end{aligned}
\end{equation}
where 
\begin{equation*}
\Lambda(\tilde{m}):= \left( \begin{matrix}
\lambda_1(\tilde{m})  & 0 \\[0.3em] 0 & \lambda_2(\tilde{m})
\end{matrix} \right).
\end{equation*}
We have thus proved the following statement.
\begin{prop}
\label{prop:specrep}
	The spectrum of the problem \eqref{me1}--\eqref{me3} is the union of the following sets.
	\begin{enumerate}
		\item { The elements of $\{ \alpha_k : k \in \mathbb{Z} \}$  such that at least one of the  corresponding $r^k$ has zero mean over $Q.$ These are eigenvalues of infinite multiplicity and the corresponding eigenfunctions $H^0(x,y)$ are of the form $w(x) r^k(y)$ for an arbitrary $w \in L^2(\mathbb{T})$. } 
		\item The set $\bigl\{\omega^2:  \exists m\in{\mathbb Z}^3\ {\rm such\ that}\ (\ref{lim:fourfinal}){\rm\ holds}\bigr\},$ 
		with the corresponding eigenfunctions {$H^0(x,y)$} of \eqref{me1}--\eqref{me3}
		having the form {$u(x) + \nabla_y v(x,y) +z(x,y),$} where {$u(x) = \exp(2\pi{\rm i} m\cdot x)\hat{u}(m)$} is an eigenfunction of 
		macroscopic problem (\ref{me7}) 
		and 
{		\begin{equation*}
		\nabla_y v(x,y)+z(x, y)=\omega^2 B(y)u(x)\ \ {\rm a.e.}\ (x,y)\in{\mathbb T}\times Q,
		\end{equation*}
that is	 $ H^0(x,y) =\big( I + \omega^2B(y) \big)\exp(2\pi{\rm i} m\cdot x)\hat{u}(m)$.}
	\end{enumerate}
\end{prop}
{ An immediate consequence of the above analysis is the following result.}
\begin{corollary}
If the matrix $\Gamma(\omega)$ is negative-definite, the value $\lambda = \omega^2$ does not belongs to the spectrum of  \eqref{me1}--\eqref{me3}.
\end{corollary}
\begin{proof}
Since $\mathcal{M}$ admits the spectral decomposition $C'(\tilde{m})\Lambda'(\tilde{m})C'(\tilde{m})^\top,$ where
\begin{equation*}
C'(\tilde{m}):=  
\left( \begin{matrix}
\tilde{e}_{11}(\tilde{m}) & \tilde{e}_{12}(\tilde{m}) & \tilde{e}_{13}(\tilde{m})\\[0.3em]\tilde{e}_{21}(\tilde{m}) & \tilde{e}_{22}(\tilde{m}) & \tilde{e}_{23}(\tilde{m}) \\[0.3em]
\tilde{m}_{1} & \tilde{m}_{2} & \tilde{m}_{3} 
\end{matrix} \right), \qquad 
\Lambda'(\tilde{m}):= \left( \begin{matrix}
\lambda_1(\tilde{m})  & 0 & 0\\[0.3em] 0 & \lambda_2(\tilde{m}) & 0 \\[0.3em] 0 & 0 & 0
\end{matrix} \right),
\end{equation*}
 a necessary condition for pairs $(m,\omega)$ such that \eqref{limfour} has a solution is 
as follows:
\begin{equation*}
{\rm det}\big(\vert m\vert^2 \Lambda'(\tilde{m})- C'(\tilde{m}) \Gamma(\omega)C'(\tilde{m})\big)=0.
\label{det_condition_1}
\end{equation*}
This is not possible since $\Lambda'(\tilde{m})$  {is positive-semidefinite and, by assumption, the matrix $\Gamma(\omega)$ and, consequently, 
the matrix $C'(\tilde{m})\Gamma(\omega)C'(\tilde{m})$ are negative-definite}.
\end{proof}

\subsection{Examples of different {admissible} wave propagation regimes {for the effective  spectral problem}}

In this section we explore the effective wave propagation properties of high-contrast electromagnetic media.  
We demonstrate that the sign-indefinite nature of the matrix-valued function $\Gamma$ gives rise to phenomena not present in the case of polarised waves. 

{  Suppose that the inclusion is symmetric under a rotation by $\pi$ around at least two of the three coordinate axes, then the matrices $A^{\rm hom}$ and $\Gamma(\omega)$ are diagonal (see Appendix): 
$A^{\rm hom}={\rm diag}(a_1,a_2,a_3),$ $\Gamma(\omega)={\rm diag}\bigl(\beta_1(\omega),\beta_2(\omega),\beta_3(\omega)\bigr)$. Here $a_i$ are positive constants and $\beta_i$ are real-valued scalar functions.
Notice that, since $| \tilde{m} | = 1$, the eigenvalues $\lambda_{1,2}(\tilde{m})$ of ${\mathcal M}(\tilde{m})$ are the solutions to the quadratic equation
\begin{equation}
\label{characteristic}
\lambda^2-\lambda\bigl\{(a_2 + a_3)\tilde{m}_1^2+ (a_1 + a_3)\tilde{m}_2^2+(a_1 + a_2)\tilde{m}_3^2)\bigr\}
+\bigl(a_1a_2\tilde{m}_3^2+a_2a_3\tilde{m}_1^2+a_1a_3\tilde{m}_2^2\bigr) =0.
\end{equation}
We will now solve the eigenvalue problem \eqref{limfour}, equivalently (\ref{lim:fourfinal}), for particular examples of such inclusions.}

\subsubsection{Isotropic propagation (no ``weak'' band gaps)} 
\label{no_weak_gaps}
If the inclusion $Q_0$ is symmetric by a $\pi/2$ rotation around at least two of the three axes, say $x_1$ and 
$x_2,$ then $a = a_1 = a_2=a_3$ and $ \beta(\omega) = \beta_1(\omega) = \beta_2(\omega)=\beta_3(\omega)$. The equation \eqref{characteristic} takes the form 
$(\lambda - a)^2 = 0,$
and therefore $\lambda_1(\tilde{m}) = \lambda_2( \tilde{m}) = a$ is an eigenvalue of multiplicity two of ${\mathcal M}(\tilde{m}),$  with orthonormal eigenvectors
given by
\begin{equation}
\tilde{e}_1(\tilde{m})=e_2,\ \ \ \tilde{e}_2(\tilde{m})=e_3\ \ \ \ \ \ \ {\rm if}\ \ |\tilde{m}_1|=1,
\label{trivial}
\end{equation} 
and 
\begin{equation}
\tilde{e}_1( \tilde{m}) = 
\frac{1}{\sqrt{1-\tilde{m}_1^2}}
{e}_1 \times \tilde{m},\ \ \ \ \ \ \ 
\tilde{e}_2(\tilde{m})= 
\frac{1}{\sqrt{1-\tilde{m}_1^2}}( {e}_1 \times \tilde{m}  ) \times \tilde{m}\ \ \ \ \ \ \ \ {\rm if}\ \ |\tilde{m}_1|<1.
\label{eigenvector_formula}
\end{equation}
As before, $e_j,$ $j=1,2,3,$ are the Euclidean basis vectors. The system \eqref{lim:fourfinal} takes the form
$$
 a \vert m \vert^2 \tilde{u}(\tilde{m})=  \beta(\omega)  \tilde{u}(\tilde{m}) , \qquad \qquad \qquad  \alpha(\tilde{m})\beta(\omega) = 0.
$$
Notice that if $\omega$ is a zero of $\beta$ then necessarily $\tilde{u}(\tilde{m})$ is the zero vector. 
For such values of $\omega,$ the above system is satisfied for any $\alpha(\tilde{m}),$ {\it i.e.} the non-trivial eigenvectors to \eqref{limfour} are parallel to $\tilde{m}$. On the other hand, if $\beta(\omega) \neq 0$, then $\alpha(\tilde{m})=0$ and 
$\omega$ is an eigenvalue of \eqref{limfour} if and only if it solves the equation $\beta(\omega) =  a\vert m \vert^2.$
In this case $\tilde{u}(\tilde{m})$ is an arbitrary element of ${\mathbb R}^2$ and $\hat{u}(m)= C(\tilde{m})^\top\tilde{u}(\tilde{m})$ is an arbitrary vector of the (2-dimensional) eigenspace spanned by the vectors $\tilde e_1(\tilde m)$ and $\tilde e_2(\tilde m).$ Finally, there are no non-trivial solutions $\hat{u}$ when $\beta(\omega)<0.$

\subsubsection{Directional propagation (existence of ``weak'' band gaps)}
If the inclusion $Q_0$ is symmetric by a $\pi/2$ rotation around one of the three coordinate axis, say $x_1,$ and by a $\pi$ rotation around another axis, say $x_2,$ one has $a = a_1$, $b = a_2=a_3$ and $\beta_2(\omega)=\beta_3(\omega)$. {  Here, recalling $| \tilde{m} | =1$, \eqref{characteristic} takes the form 

$$
(\lambda - b )(\lambda  -  a(1 - \tilde{m}_1^2)- b\tilde{m}_1^2) = 0,
$$
whence $\lambda_1(\tilde{m}) = a(1-\tilde{m}^2_1)  + b\tilde{m}^2_1 $, $\lambda_2(\tilde{m}) = b$.} There are now two separate cases to consider.

{\it Case 1).} Assume that  { $ |\tilde{m}_1 | = 1$}, {\it i.e.} the vector $\tilde{m}$ is parallel to the axis of higher symmetry.  Here, $\mathcal{M}(\tilde{m})= {\rm diag}(0,b,b)$ and $b$ is an eigenvalue of multiplicity two with the eigenspace spanned by the vectors (\ref{trivial}). 
The system \eqref{lim:fourfinal} takes the form
$$
 bm_1^2\tilde{u}(\tilde{m})=\beta_2(\omega)\tilde{u}(\tilde{m}), \quad\quad \alpha(\tilde{m})\beta_1(\omega) = 0.
$$
Here, if $\beta_2(\omega)<0,$ then necessarily $\tilde{u}(\tilde{m})=0$ and non-trivial solutions $\hat{u}(m)=\alpha(\tilde{m})m$ exist if and only if $\beta_1(\omega)=0.$
On the other hand, if $\beta_1(\omega)<0,$ then necessarily $\alpha(\tilde{m})=0$ and non-trivial solutions 
$\hat{u}(m)=C(\tilde{m})^\top\tilde{u}(\tilde{m})$ exist if and only if $\beta_2(\omega)>0.$ The first situation only occurs at a discrete set 
of values $\omega,$ while {, unlike in the isotropic case,} the second situation can give rise to intervals of admissible $\omega$ with a reduced number of propagating modes, which we refer to as ``weak band gaps''.


{\it Case 2).}  {  Assume $|\tilde{m}_1|<1$,} {\it i.e.} the vector $\tilde{m}$ is not parallel to the axis of higher symmetry.   { Recall that the eigenvectors corresponding to $\lambda_1(\tilde{m}),$ $\lambda_2(\tilde{m})$ are given by $\tilde{e}_1(\tilde{m}),$ 
$\tilde{e}_2(\tilde{m})$ in (\ref{eigenvector_formula}).  By setting $\Gamma(\omega) = {\rm diag}(\beta_1(\omega) - \beta_2(\omega), 0,0) + \beta_2(\omega) I$ it is easy to see that the} system  \eqref{lim:fourfinal} takes the form
\begin{equation}
\begin{aligned}
 | m |^2 \lambda_1(\tilde{m}) \tilde{u}_1(\tilde{m})& = \beta_2(\omega) \tilde{u}_1(\tilde{m}), \\[0.3em]
\bigl( \beta_1(\omega)- \beta_2(\omega)\bigr)m_1 \sqrt{1-m^2_1} \tilde{u}_2(\tilde{m})& = \alpha \left( (\beta_1(\omega) -\beta_2(\omega)) m_1^2 + \beta_2(\omega)\right), \\[1pt]
| m |^2 \lambda_2(\tilde{m}) \tilde{u}_2(\tilde{m}) =  \big(  \beta_2(\omega) +  \bigl(\beta_1(\omega) - & \beta_2(\omega))(1 - m^2_1\bigr) \big)  \tilde{u}_2(\tilde{m}) - \alpha {  (\beta_1(\omega) - \beta_2(\omega))}m_1 \sqrt{1 - m_1^2},
\end{aligned} 
\label{lambda_system}
\end{equation}
If $\tilde{m}_1 = 0$, {\it i.e.} the vector $\tilde{m}$ is perpendicular to the direction of higher symmetry, then the system (\ref{lambda_system}) fully decouples and reduces to
$$
 \vert m \vert^2 a \tilde{u}_1(\tilde{m}) = \beta_2(\omega) \tilde{u}_1(\tilde{m}) ,\hspace{1.4cm}
 \vert m\vert^2 b \tilde{u}_2(\tilde{m}) = \beta_1(\omega)\tilde{u}_2(\tilde{m}) , \hspace{1.4cm} \alpha \beta_2(\omega) = 0.
$$
Suppose,  $\beta_1(\omega)$ ({\it resp.} $\beta_2(\omega)$) is negative for some $\omega$, then the above system implies that $\tilde{u}_2(\tilde{m}) = 0$ ({\it resp.} $\tilde{u}_1(\tilde{m})= 0$). In this case, we see that propagation is restricted solely to the direction of $\tilde{e}_1(\tilde{m})$ ({\it resp.} $\tilde{e}_2(\tilde{m})$) which is orthogonal to the eigenvector(s) corresponding to the negative eigenvalue 
of $\Gamma(\omega)$. In both situations weak band gaps are present.



\begin{remark}
	Recently, there has been several works on the analysis of problems with ``partial'' or ``directional" wave propagation in the context of elasticity, where at some frequencies, propagation occurs for some but not for all values of the wave vector: the analysis of the vector problems for thin structures of critical thickness \cite{ZhPast2002}, the analysis of  high-contrast \cite{ZhPast2013}, and partially high-contrast \cite{SmyshlyaevPartial} periodic elastic composites. To our knowledge, the effect we describe here is the first example of a similar kind for Maxwell equations.  
\end{remark}

\begin{remark}
	When the ``size'' $T$ of the domain ${\mathbb T}$ increases to infinity, the spectrum of \eqref{me1}--\eqref{me3} converges to a union of intervals (``bands'') separated by intervals
	of those values $\omega^2$ for which the matrix $\Gamma(\omega)$ is negative-definite (``gaps'', or ``lacunae'').
	As above, we say that $\omega^2$ belongs to a weak band gap (in the spectrum of \eqref{me1}--\eqref{me3}) if at least one eigenvalue of $\Gamma(\omega)$ is   {positive-semidefinite} and at least one eigenvalue of $\Gamma(\omega)$ is negative.  
\end{remark}

\section{Two-scale asymptotic expansion   { of the eigenfunctions}}
\label{expansionsection}

Here we give the details of the recurrent procedure for the construction of the series (\ref{e1.2}). As before

Substituting the expansion \eqref{e1.2} into \eqref{e1.1} and equating coefficients in front of  $\eta^{-2}, \eta^{-1},$  and $\eta^0,$ we arrive the following sets of equations, where $x\in{\mathbb T}$ is a parameter:
\begin{eqnarray}
\curl{y}\bigl(\epsilon_1^{-1}(y){\curl{y}{H^0(x,y)}}\bigr)= 0, & y \in Q_1, \label{e1.3} \\[0.4em]
\epsilon_1^{-1}(y)\curl{y}{H^0}(x,y) \times n(y) \big\vert_{+} = 0, & y \in \partial{Q_0}, \label{e1.4}
\end{eqnarray}
\begin{eqnarray}
\curl{y}\bigl(\epsilon_1^{-1}(y){\curl{y} H^1(x,y)}\bigr) =   {  - \bigl( \curl{y}\epsilon_1^{-1}(y){\curl{x} + \curl{x}\epsilon_1^{-1}(y)\curl{y}\bigr) }H^0(x,y)}, \quad  y \in Q_1, \label{e2.1} \\[0.4em]
\left(\epsilon_1^{-1}(y)\curl{y}{H^1}(x,y)\times n(y)  + \epsilon_1^{-1}(y)\curl{x}{H^0}(x,y)\times n(y)\right) \big\vert_{+}= 0, \ \ \ \quad\qquad y \in \partial{Q_0}, \label{e2.2}
\end{eqnarray}
\begin{eqnarray}
\begin{split}
\curl{y}\bigl(\epsilon_1^{-1}(y){\curl{y} H^2(x,y)}\bigr) = & - \left( \curl{y}\epsilon_1^{-1}(y){\curl{x}{}} + \curl{x}\epsilon_1^{-1}(y){\curl{y}{}} \right)H^1(x,y) \\[0.4em]  &- \curl{x}\epsilon_1^{-1}(y){\curl{x} H^0(x,y)} + \omega^2 H^0(x,y), \quad  y \in Q_1, 
\end{split} 
\label{e3.1}\\[0.4em]
\begin{split}
\left(\epsilon_1^{-1}(y)\curl{y}{H^2(x,y)}\right.+&\left.\epsilon_1^{-1}(y)\curl{x}{H^1}(x,y)\right)\times n(y) \big\vert_{+}
\\[0.4em]
&
= \epsilon_0^{-1}(y)\curl{y}{H^0}(x,y) \times n(y) \big\vert_{-}, \ \quad\quad y \in \partial{Q_0},
\end{split}
\label{e3.2}
\end{eqnarray}
and
\begin{eqnarray}
\curl{y}\bigl(\epsilon_0^{-1}(y){\curl{y} H^0(x,y)}\bigr)=  \omega^2 H^0(x,y), \quad  y \in Q_0, \label{e4.1}\\[0.4em]
H^0(x,y)\bigr\vert_-=H^0(x,y)\bigr\vert_+, \quad y\in\partial Q_0.\label{e4.2}
\end{eqnarray}
  { Multiplying the equation \eqref{e1.3} by $H^0$, integrating by parts over $Q_1$, and using \eqref{e1.4} }shows that ${\rm curl}_yH^0(x,y)=0,$ $y\in Q_1.$ More precisely, for all $x\in{\mathbb T}$ we seek $H^0(x,\cdot)$ from the space ({\it cf.} (\ref{spaceV}))
\begin{equation*}
V : = \left\{ v \in [H^1_{\#}(Q)]^3\,\big\vert\,\curl{}{v} = 0 \text{ in $Q_1$} \right\}.
\end{equation*}
Before proceeding, we recall a characterisation of the space $V$ (see \cite{maxCh}) that proves useful in the analysis of the term $H^0$.
\begin{lem}[Characterisation of $V$] 
\label{lem:spaceV}
A function $v\in[H^1_{\#}(Q)]^3$ is an element of the space $V$ if and only if
$$
v(y)= a + \nabla b(y) + c(y)\ \ \ \ \ {\rm a.e.}\ y\in Q,
$$
	for some $a \in \mathbb{R}^3$, $  {b \in H^2_{\#}(Q)}$, $c \in [H^1_0(Q_0)]^3$.
\end{lem}

Taking into account, via Lemma \ref{lem:spaceV}, that the leading-order term $H^0$ is of the form 
\begin{equation}
\label{H0}
H^0(x,y) = u(x) + \nabla_y v(x,y) + z(x,y)
\end{equation}
and substituting \eqref{H0} into the equations \eqref{e2.1}--\eqref{e2.2}, we find that the coefficient  $H^1$ has the 
representation
$H^1(x,y) = N(y)\,{\rm curl}\,u(x) + \widetilde{H}^1(x,y),$ up to the addition of an element of $V$. Here the term $\widetilde{H}^1(x,y)$ satisfies 
\begin{eqnarray}
\curl{y}\Bigl(\epsilon_1^{-1}(y){\big( \curl{y}\widetilde{H}^1(x,y) + \curl{x}{\nabla_y v(x,y)}\big)}\Bigr) =  0, &  y \in Q_1, \label{e2.5} \\
\epsilon_1^{-1}(y)\big( \curl{y}{\widetilde{H}^1}(x,y) + \curl{x}{\nabla_y v(x,y)}  \big) \times n(y) \big\vert_{+}  = 0, & y \in \partial{Q_0}, \label{e2.6}
\end{eqnarray}
and $N=N(y)$ is a $Q$-periodic matrix-valued function whose columns $N^r=N^{r}(y)$, $r = 1, 2, 3,$ are solutions 
to the problems
\begin{equation}
\curl{}\Bigl(\epsilon_1^{-1}(y){\bigl(\curl{}  N^r(y) + e^r\big)}\Bigr)=  0,\ \   y \in Q_1,\ \ \ \ \ \ 
\epsilon_1^{-1}(y)\big(\curl{}{ N^r(y)} + e^r \big) \times n(y)  = 0, \ \  y \in \partial{Q_0},  \label{e2.4}
\end{equation}
where $e^r$ is the $r$th Euclidean basis vector. It is shown (\cite{Cooper}, \cite{KamSmysh2013}) that (\ref{e2.4}) admits a unique solution in $V^\perp,$ the orthogonal complement to $V$ in the 
space $[H^1_{\#}(Q)]^3.$ 

Looking for $H^1(x,\cdot)\in[H^1_{\#}(Q)]^3$ and taking into account the identity $\curl{x}{\nabla_y} = - \curl{y}{\nabla_x}$ together with \eqref{e2.5}--\eqref{e2.6}, we infer that for all $x\in{\mathbb T}$ the function $h(x,\cdot) : = \widetilde{H}^1(x,\cdot) - \nabla_x v(x,\cdot)$ is a solution in $[H^1_{\#}(Q)]^3$ 
to
\begin{eqnarray*}
	\curl{y}\bigl(\epsilon_1^{-1}(y){\curl{y}{h(x,y)}} \bigr)=  0, &  y \in Q_1,\ \ \ \ \ \ \ \ \epsilon_1^{-1}(y)\curl{y}{h(x,y)}  \times n(y) \big\vert_{+}  = 0, & y \in \partial{Q_0}. 
\end{eqnarray*}
In particular, the function $h$ belongs to the space $V$. Therefore,  one has
\begin{equation}
\label{H1}
H^1(x,y) = N(y)\,\curl{x}{u}(x) + \nabla_x v(x,y),
\end{equation} up to the addition of an element of $V$. (As we discuss in Remark \ref{last_remark} below, one can specify 
the divergence ${\rm div}_yH^1(x,y).$ This, along with the condition that the $y$-average of $H^1$ vanishes, defines this additional element of $V$ in a unique way.)

Further, multiplying the equation \eqref{e3.1} by an arbitrary test function $\phi\in V$ and integrating over $Q_1$ yields
\begin{gather}
\int_{Q_1}\curl{y}\bigl(\epsilon_1^{-1}(y){
\curl{y}{H^2(x,y)}}\bigr)\cdot\phi(y) \,\mathrm{d}y 
= \int_{Q_1} \omega^2H^0(x,y)\cdot\phi(y)\,\mathrm{d}y - \int_{Q_1} \curl{y}\bigl(\epsilon_1^{-1}(y){\curl{x}{H^1(x,y)}}\bigr)\cdot\phi(y)\,\mathrm{d}y \nonumber \\ 
- \left( \int_{Q_1} \curl{x}\bigl(\epsilon_1^{-1}(y){\curl{x}{H^0(x,y)}}\bigr)\cdot\phi(y) + \curl{x}\bigl(\epsilon_1^{-1}(y){\curl{y}{H^1(x,y)}}\bigr)\cdot\phi(y)\,\mathrm{d}y \right). \label{e5.1}
\end{gather}
We integrate by parts in the left-hand side of \eqref{e5.1}   { to determine that}
\begin{flalign}
\label{e5.2}
\int_{Q_1} \curl{y}\bigl(\epsilon_1^{-1}(y){\curl{y}{H^2(x,y)}}\bigr)\cdot\phi(y)\,\mathrm{d}y &=
\int_{\partial{Q_0}}\epsilon_1^{-1}(y)\bigl(\curl{y}{H^2(x,y)} \times n(y) \big\vert_{+}\bigr)\cdot\hspace{1pt} \phi(y) \ \mathrm{d}y.
\end{flalign}
Now we perform integration by parts in the individual terms in the right-hand side of \eqref{e5.1}. 
\begin{flalign}
- \int_{Q_1} & \curl{y}\bigl(\epsilon_1^{-1}(y){\curl{x}{H^1(x,y)}}\bigr)\phi(y)\,\mathrm{d}y  =  - \int_{\partial{Q_0}}\bigl(\epsilon_1^{-1}(y)
\curl{x}{H^1(x,y)} \times n(y) \big\vert_{+}\bigr)\hspace{1pt}\cdot\phi(y)\,\mathrm{d}y \nonumber \\
\overset{\text{by \eqref{e3.2}}}{=} &  \int_{\partial{Q_0}}\Bigl\{\bigl(\epsilon_1^{-1}(y)(\curl{y}{H^2(x,y)} \times n(y) \big\vert_{+}\bigr)\hspace{1pt}\cdot\phi(y) + \bigl(\epsilon_0^{-1}(y)\curl{y}{H^0(x,y)} \times n(y) \big\vert_{-} \bigr)\hspace{1pt}\cdot
\phi(y)\Bigr\}\,\mathrm{d}y \nonumber \\
= & \int_{\partial{Q_0}}\bigl(\epsilon_1^{-1}(y)\curl{y}{H^2(x,y)} \times n(y) \big\vert_{+}\bigr)\hspace{1pt}\cdot\phi(y)+
\int_{Q_0} \curl{y}\bigl(\epsilon_0^{-1}(y){\curl{y}{H^0(x,y)}}\bigr)\cdot\phi(y)\,\mathrm{d}y\nonumber\\
&-\int_{Q_0}\epsilon_0^{-1}(y)\curl{y}{H^0(x,y)}\cdot\curl{y}{\phi(y)}\,\mathrm{d}y \nonumber \\
\overset{\text{by \eqref{e4.1}}}{=} & \int_{\partial{Q_0}}\bigl(\epsilon_1^{-1}(y)\curl{y}{H^2(x,y)} \times n(y) \big\vert_{+}\bigr)\hspace{1pt}\cdot\phi(y) + \int_{Q_0} \omega^2H^0(x,y)\cdot\phi(y)\,\mathrm{d}y\nonumber\\ 
&- \int_{Q_0} \epsilon_0^{-1}(y)\curl{y}{H^0(x,y)}\cdot\curl{y}{\phi(y)} \,\mathrm{d}y \label{e5.3}
\end{flalign}
Taking into account the representations \eqref{H0} and \eqref{H1}, we find that
\begin{flalign}
\int_{Q_1} &\Bigl\{ \curl{x}\bigl(\epsilon_1^{-1}(y){\curl{x}{H^0(x,y)}}\bigr)\cdot\phi(y) + \curl{x}\bigl(\epsilon_1^{-1}(y){\curl{y}{H^1(x,y)}}\bigr)
\cdot\phi(y)\Bigr\}\,\mathrm{d}y \nonumber \\
& = \int_{Q_1} \curl{x}\Bigl\{\epsilon_1^{-1}(y){\Bigl(\big(I+N(y)\big) \curl{x}{u(x)} + \curl{x}{\nabla_y v(x,y)} + 
	\curl{y}{\nabla_x v(x,y)}\Bigr)}\Bigr\}\cdot\phi(y)\,\mathrm{d}y \nonumber \\
& = \int_{Q_1} \curl{x}\Bigl\{\epsilon_1^{-1}(y){\Big(\big( I + N(y) \big) \curl{x}{u(x)} \Big)}\Bigr\}\cdot\phi(y)\,\mathrm{d}y, \label{e5.4}
\end{flalign}
where we again make use of the identity $\curl{x}{\nabla_y}=-\curl{y}{\nabla_x}.$ Finally, equations \eqref{e5.1}--\eqref{e5.4} imply 
\begin{multline}
\int_{Q_1} \curl{x}\Bigl\{\epsilon_1^{-1}(y){\Bigl(\big(I + N(y)\big)\curl{x}{u(x)}\Bigr)}\Bigr\}\cdot\phi(y)\,\mathrm{d}y + \int_{Q_0} \epsilon_0^{-1}(y)\curl{y}{H^0(x,y)}\cdot\curl{y}{\phi(y)}\,\mathrm{d}y \\ = \int_{Q} \omega^2H^0(x,y)\cdot\phi(y)\,\mathrm{d}y. \ \ \ \ \ \ \ \  { \forall \phi \in V } \label{hom.1}
\end{multline}

In what follows we derive the system (\ref{me1})--(\ref{me3}) by considering different choices of the test function $\phi$ in the identity \eqref{hom.1}.

Step 1. Choosing test functions $\phi \in\bigl[C^\infty_0(Q_0)\bigr]^3$ in \eqref{hom.1} we find that
$$
\curl{y}\bigl(\epsilon_0^{-1}(y){\curl{y}{H^0(x,y)}}\bigr)= \omega^2 H^0(x,y) \qquad y \in Q_0.
$$
Using the representation \eqref{H0} and the identity $\curl{y}{\nabla_y} = 0$, we arrive at  \eqref{me3}.

Step 2.  { Choosing $\phi = \nabla_y\psi$ in \eqref{hom.1}, performing integration by parts, using the identity $\div{y}{\curl{x}{}} = -\div{x}{\curl{y}{}}$ and recalling \eqref{e2.4} gives}
\begin{flalign*}
\int_{Q} \omega^2H^0(x,y)\cdot\nabla_y\psi\,\mathrm{d}y & =\int_{Q_1} \curl{x}{\Bigl\{\epsilon_1^{-1}(y)\Big(\big(I + N(y) \big) \curl{x}{u(x)} \Big)\Bigr\}}\cdot\nabla_y \psi(y)\,\mathrm{d}y \\
& =  \int_{Q_1} \div{y}{ \curl{x}{\Bigl\{\epsilon_1^{-1}(y)\Big(\bigl(I + N(y)\bigr)\curl{x}{u(x)}\Big)}\Bigr\}}\cdot\psi(y) \,\mathrm{d}y 
\\
& = -\int_{Q_1} \div{x}{ \curl{y}{\Bigl\{\epsilon_1^{-1}(y)\Big(\bigl( I + N(y)\bigr)\curl{x}{u(x)} \Big)}\Bigr\}}\cdot\psi(y)\,\mathrm{d}y= 0.
\end{flalign*} 
Therefore, we deduce that
\begin{equation}
\div{y}{H^0(x,y)} = 0, \qquad y \in Q,
\label{H0solenoidal}
\end{equation}
and taking into account \eqref{H0} we obtain the equation \eqref{me2}. 

Step 3. Choosing $\phi(y) \equiv 1$ in the identity\eqref{hom.1} we find, using the representation  \eqref{H0} once   {more}, that \eqref{me1} holds, where the matrix $A^{\rm hom}$ emerges as the result of integrating the expression ${\rm curl}_y N(y)+I$ with respect to 
$y\in Q_1.$ 

In the next section we use the above formal construction of the series (\ref{e1.2}) to justify the two claims of Theorem 
\ref{maintheorem}.

\section{Proof of Theorem \ref{maintheorem}}



For each $\eta>0,$ denote by ${\mathcal A}_\eta$ the operator in the space\footnote{We denote by $L^2_{\rm {\#}sol}(\mathbb T)$ the closure of the set of smooth divergence-free vector fields on  ${\mathbb T}$ with respect to the $L^2(\mathbb T)$-norm.}
 $L^2_{\rm {\#}sol}(\mathbb T)$ defined in a standard way by the bilinear form ({\it cf.} (\ref{e1.1})) 
{\[
\int_{\mathbb T}\varepsilon_\eta^{-1}\bigl(\tfrac{\cdot}{\eta}\bigr)\,{\rm curl}\,u\cdot {\rm curl}\,v,\ \ \ \ \ \  u,v\in [H^1_{\#}({\mathbb T})]^3\cap L^2_{\rm {\#}sol}(\mathbb T)
=:{\mathcal H}.
\]}
{ For fixed $\omega$ in the spectrum of \eqref{me1}--\eqref{me2}, let $H^0$ be a corresponding eigenfunction. Consider the (unique) solution $\widetilde{H}^\eta\in{\mathcal H}$ to the problem
\begin{equation}
({\mathcal A}_\eta+I)\widetilde{H}^\eta=(\omega^2+1)H^0(\cdot, \tfrac{\cdot}{\eta}).
\label{Htildeeq}
\end{equation}
Denote also 
\[
{\mathfrak b}_\eta(u,v):=
\int_{\mathbb T}\varepsilon_\eta^{-1}\bigl(\tfrac{\cdot}{\eta}\bigr)\,{\rm curl}\,u\cdot {\rm curl}\,v+
\int_{\mathbb T}u\cdot v,\ \ \ \ u,v\in [H^1_{\#}({\mathbb T})]^3,
\]
and ({\it cf.} (\ref{e1.2}))
\begin{equation}
H^{(2)}(\cdot, \eta):=H^0\bigl(\cdot, \tfrac{\cdot}{\eta}\bigr)+\eta H^1\bigl(\cdot, \tfrac{\cdot}{\eta}\bigr)+\eta^2 H^2\bigl(\cdot, \tfrac{\cdot}{\eta}\bigr),
\label{H2eqn}
\end{equation}
where $H^j,$ $j=1, 2,$ are solution of the system of recurrence relations described in Section \ref{expansionsection}. The existence of solutions $H^1$, $H^2$ is guaranteed by a result we established in \cite[Lemma 3.4]{maxCh}. As these solutions are unique up to the addition of an element from $V$, we shall choose them as in  Remark \ref{last_remark}. }

\begin{prop} There exists a constant $\widehat{C}>0$ such that the estimate 
	\begin{equation}	{\mathfrak b}_\eta\bigl(\widetilde{H}^\eta-H^{(2)}(\cdot, \eta), \varphi\bigr)\le\widehat{C}\eta\sqrt{{\mathfrak b}_\eta(\varphi, \varphi)}
	\label{bform_est}
	\end{equation}
	holds for all { $\varphi\in [H^1_{\#}(\mathbb T)]^3.$}
\end{prop}

\begin{proof}
	Using the definition of the function $\widetilde{H}^\eta$ and the recurrence relations (\ref{e1.3})--(\ref{e4.2}) 
	yields
{
	\begin{equation}
		\label{remainders.0}
		\begin{aligned}
	{\mathfrak b}_\eta\bigl(  \widetilde{H}^\eta-H^{(2)}(\cdot, \eta), & \varphi\bigr)  = \int_{\mathbb T}\varepsilon_\eta^{-1}\bigl(\tfrac{\cdot}{\eta}\bigr){\rm curl}\,\widetilde{H}^\eta\cdot{\rm curl}\varphi+ \int_{\mathbb T}\widetilde{H}^\eta\cdot\varphi \\[0.5em]
	&	-\int_{\mathbb T}\varepsilon_\eta^{-1}\bigl(\tfrac{\cdot}{\eta}\bigr){\rm curl}\Bigl(H^0\bigl(\cdot, \tfrac{\cdot}{\eta}\bigr)+\eta H^1\bigl(\cdot, \tfrac{\cdot}{\eta}\bigr)+{\eta^2} H^2\bigl(\cdot, \tfrac{\cdot}{\eta})\Bigr)\cdot{\rm curl}\,\varphi \\[0.5em]
	& 
	-\int_{\mathbb T}\Bigl(H^0\bigl(\cdot, \tfrac{\cdot}{\eta}\bigr)+\eta H^1\bigl(\cdot, \tfrac{\cdot}{\eta}\bigr)+{ {\eta^2 }}H^2\bigl(\cdot, \tfrac{\cdot}{\eta}\bigr)\Bigr)\cdot\varphi \\[0.5em]
	& = \int_{\mathbb{T}} F^1( \cdot , \eta) \cdot \phi + \int_{\mathbb{T}} F^2(\cdot , \eta) \cdot \eta\,\curl{}{\phi}.
			\end{aligned}
	\end{equation}
Here,  $F^1, F^2$ are elements of $L^2(\mathbb{T})$ defined for a.e. $x \in \mathbb{T}$ by 
\begin{eqnarray}
F^1(x,\eta)= -\eta \Bigl(  \chi_0(y){\rm curl}_x\bigl(\epsilon_0^{-1}(y){\rm curl}_y H^0(x, y)\bigr)\ \ \ \ \ \ \ \ \ \ \ \ \ \ \ \ \ \ \ \ \ \ \ \ \ \ \ \ \ \ \ \ \ \ \ \ \ \ \ \ \ \ \ \ \ \ \ \ \ \ \ \ \ \ \ \ \ \ \ \ \ \nonumber \\[0.4em]
\ \ \ \ \ \ \ \ +\chi_1(y)\Big\{{\rm curl}_x\bigl(\epsilon_1^{-1}(y){\rm curl}_x H^1(x,y)\bigr)+{\rm curl}_x\bigl(\epsilon_1^{-1}(y){\rm curl}_y H^2(x,y)\bigr)\Big\} +H^1(x,y)+\eta H^2(x,y)\Bigr)\Bigr\vert_{y=\tfrac{x}{\eta}},\nonumber \\[0.4em]
F^2(x,\eta)=-\eta\Bigl(\chi_0(y)\epsilon_0^{-1}(y)\big\{{\rm curl}_x H^0(x,y)+{\rm curl}_y H^1(x,y)+\eta\,{\rm curl}_x H^1(x,y)\ \ \ \ \ \ \ \ \ \ \ \ \ \ \ \ \ \ \ \ \ \ \  \ \nonumber \\[0.4em]
+\eta\,{\rm curl}_y H^2(x,y) + {\eta^2\,{\rm curl}_x H^2(x,y)}\big\} +\chi_1(y)\epsilon_1^{-1}(y){\rm curl}_xH^2(x,y)\Bigr)\biggr\vert_{y=\tfrac{x}{\eta}}.\ \ \ \ \ \ \ \ \ \ \ \ \ \ \ \ \ \ \ \ \ \ \ \ \ \ \label{remainders}
\end{eqnarray}	
Notice that the functions $H^0=H^0(x,y), H^1=H^1(x,y), H^2=H^2(x,y)$ all belong to the space 
	$C^\infty_{\#}\bigl({\mathbb T}, H^1_{\#}(Q)\bigr)$. Indeed, this is seen to be true for $H^0$ by Proposition \ref{prop:specrep}; in the case of $\omega = \alpha_k$ we choose $w \in C^\infty_{\#}(\mathbb{T})$. The assertions for $H^1$ and $H^2$ now follow from  formula (\ref{H1}) for the corrector
	$H^1(x,y),$ and the boundary-value problem (\ref{e3.1})--(\ref{e3.2}) for the function $H^2(x,y).$ 	It then follows from \eqref{remainders} that $
 	|| F^1 (\cdot, \eta) ||_{L^2(\mathbb{T})} \le C \eta$, $|| F^2 (\cdot, \eta) ||_{L^2(\mathbb{T})} \le C \eta$, and by applying the H\"{o}lder inequality to \eqref{remainders.0} we deduce that  
	\[
	{\mathfrak b}_\eta\bigl(\widetilde{H}^\eta-H^{(2)}(\cdot, \eta), \varphi\bigr)\le C\eta \biggl(\int_{\mathbb T}\vert\varphi\vert^2+\int_{\mathbb T}\vert\eta\,{\rm curl}\,\varphi\vert^2\biggr)^{1/2},
	\]
	as required.}
\end{proof}

The above proposition 
implies the following statement.

\begin{thm} 
	\label{last_theorem}
	There exists a constant C such that estimate
	$\bigl\Vert\widetilde{H}^\eta-H^0(\cdot, \cdot/\eta)\bigr\Vert_{L^2(\mathbb T)}\le C\eta$
	holds for all $\eta.$
\end{thm}

\begin{proof}
	Setting $\varphi=\widetilde{H}^\eta-H^{(2)}(\cdot, \eta)$ in the estimate (\ref{bform_est}) yields
	\[
	\widehat{C}^2\eta^2\ge{\mathfrak b}_\eta\bigl(\widetilde{H}^\eta-H^{(2)}(\cdot, \eta),\,\widetilde{H}^\eta-H^{(2)}(\cdot, \eta)\bigr)
	\ge\bigl\Vert\widetilde{H}^\eta-H^{(2)}(\cdot, \eta)\bigr\Vert^2_{L^2(\mathbb T)}.
	\]
	The claim of the theorem now follows, by noting that in view of (\ref{H2eqn}) we have 
	\[
	\bigl\Vert H^{(2)}(\cdot, \eta)-H^0(\cdot, \cdot/\eta)\bigr\Vert_{L^2(\mathbb T)}\le\widetilde{C}\eta
	\]
	for some $\widetilde{C}>0,$ and hence
	\[
	\bigl\Vert\widetilde{H}^\eta-H^0(\cdot, \cdot/\eta)\bigr\Vert_{L^2(\mathbb T)} 
	\le\bigl\Vert\widetilde{H}^\eta-H^{(2)}(\cdot, \eta)\bigr\Vert_{L^2(\mathbb T)}+\bigl\Vert H^{(2)}(\cdot, \eta)-H^0(\cdot, \cdot/\eta)\bigr\Vert_{L^2(\mathbb T)}\le(\widehat{C}+\widetilde{C})\eta,
	\]
	as required.
\end{proof}




The claims of Theorem \ref{maintheorem} now follow from the estimate
\begin{equation}
\bigl\Vert\bigl((\omega^2+1)^{-1}-({\mathcal A}_\eta+I)^{-1}\bigr)H^0(\cdot, \cdot/\eta)\bigr\Vert_{L^2(\mathbb T)}
\le(\omega^2+1)^{-1}\bigl\Vert H^0(\cdot, \cdot/\eta)-\widetilde{H}^\eta\bigr\Vert_{L^2(\mathbb T)}\le C\eta,
\label{lastest}
\end{equation}
where we used 	the definition (\ref{Htildeeq}) of the function $\widetilde{H}^\eta$ and 	
Theorem \ref{last_theorem}. Indeed, from \cite[p.\,109]{VL}, we infer that the quantities ${\rm dist}\bigl((\omega^2+1)^{-1}, {\rm Sp}\bigl(({\mathcal A}_\eta+1)^{-1}\bigr)\bigr)$ and ${\rm dist}\bigl((\omega^2+1)^{-1}H^0(\cdot,\cdot/\eta), X_\eta\bigr)$
are controlled above by the right-hand side of (\ref{lastest}), which completes the proof of Theorem \ref{maintheorem}.

\begin{remark}
	\label{last_remark}
	Note that $H^{(2)}$ is not solenoidal in general, but can be defined in such a way that it is ``close'' to a solenoidal field, thanks to the equation (\ref{H0solenoidal}) (equivalently, (\ref{me2})) and the special 
	choice of the function $H^1$ so that 
	\[
	{\rm div}_xH^0(x,y)+{\rm div}_yH^1(x,y)=0\ \ \ {\rm a.e.}\ (x,y)\in{\mathbb T}\times Q.
	\]
	The function $H^{(2)}$ thus defined is $\eta$-close to the eigenspace $X_\eta$ in the norm of $[H^1_{\#}(Q)]^3.$
\end{remark}

\section*{Appendix: Symmetry of $A^{\rm hom}$ and $\Gamma(\omega)$ under rotations}

Suppose that $A \in\bigl[L^\infty(Q)\bigr]^{3\times3}$ is symmetric such that $A\ge\nu I$ on $Q_1,$ $\nu>0$ and $A \equiv 0$ on $Q_0$. 
Consider the matrix	
	\begin{equation*}
	A^{\rm hom}_{pq}:=\int_{Q} A \bigl( {\rm curl}{N}_{p}^q + \delta_{pq}  \bigr) \qquad  p,q \in \{1,2,3\},
	\end{equation*}
	where $N^q$ is the unique solution  to the problem ({\it cf.} \cite{KamSmysh2013}, \cite[Lemma 3.4]{maxCh} and \eqref{e2.4} above for $a = \chi_1$, the characteristic function of $Q_1$)	
	\[
	{\rm curl} \,{ \big(A[ {\rm curl}\,N^q + e_q]\big)} = 0,\ \ \ \ \ \ \ N^q\in \{ u \in [H^1_{\#}(Q)]^3 : A\,{\rm curl}\,u =0 \}^\perp.
	\]
Here the superscript ``$\perp$'' denotes the orthogonal complement in $[H^1_{\#}(Q)]^3.$ Notice that if, for fixed $\zeta \in \RR^3$, we multiply each of the above equations by $\zeta_q$ then  
	\begin{equation}
	\label{defahom}
	A^{\rm hom} \zeta =\int_{Q} A\, \bigl( {\rm curl}{N_\zeta} +  \zeta  \bigr),
	\end{equation}
where the vector $N_\zeta$, whose components are $N^q_p \zeta_q$, is the unique solution to the problem
	\begin{equation}
	\label{cellxi}
	{\rm curl}\, {\big( A[ {\rm curl} N_\zeta + \zeta]\big)} = 0,\ \ \ \ \ \ \ N_\zeta\in \{ u \in [H^1_{\#}(Q)]^3 : A\,{\rm curl}\,u =0 \}^\perp.
		\end{equation}
	It is clear that the matrix representation of the bounded linear mapping $ \zeta \mapsto \int_{Q}  A\, \bigl( {\rm curl}{N_\zeta} +  \zeta  \bigr)$ is equal to $A^{\rm hom}$. The following property holds.
\begin{prop}
\label{Ahom_property}
Suppose that $\sigma$ is a rotation such that $\sigma Q = Q$ and assume that
\begin{equation}
\label{asym}
A(y)= \sigma^{-1}A(\sigma y ) \sigma,\ \ \ \ y\in Q.
\end{equation}
Then , $A^{\rm hom}$ 
inherits the same symmetry, {\it i.e.} 
one has
\begin{equation}
\label{ahomsym}
A^{\rm hom}=\sigma^{-1}A^{\rm hom}\sigma.
\end{equation}
In particular, $A^{\rm hom}_{kl} = A^{\rm hom}_{lk} =  0$, for all $l \neq k$.
\end{prop}
\begin{proof}
For each $u\in[H^1_{\#}(Q)]^3$ let 
$w$ be the solution of the vector equation
\begin{equation*}
\bigl({\rm curl}\,{\rm curl}\, w(y)\bigr)_\alpha=\sigma_{s\alpha}\epsilon_{slm}\sigma_{mr}\frac{\partial u_l(y)}{\partial y_r},\ \ \ \ \ \alpha=1,2,3,
\end{equation*}
in the space 
\[
{\mathcal H}:=\bigl\{w\in[H^1_{\#}(Q)]^3:\ {\rm div}\,w=0,\ \ \langle w \rangle=0\bigr\}.
\]
It is clear that such a solution exists. We denote by $\hat{u}$ the vector field ${\rm curl}\, w$.  A direct calculation, using the property $\sigma^{-1}=\sigma^\top$, yields   
\begin{equation}
{\rm curl}_{y'}u(\sigma^{-1}y')=\sigma{\rm curl}\,\hat{u}(\sigma^{-1}y').
\label{inverse}
\end{equation}
Therefore, for all $\varphi\in[H^1_{\#}(Q)]^3$, the above equality and the assumption \eqref{asym} imply
\begin{flalign*}
\int_Q A(y')\curl{y'}{u(\sigma^{-1}y')} \cdot \curl{y'}{\varphi(\sigma^{-1}y')}
\,dy' 
&
= \int_Q A(\sigma y)\sigma\curl{}{\hat{u}}(y)\cdot\sigma\curl{}{\hat{\varphi}}(y) \,{\rm d}y\nonumber
\\
& =  \int_Q A(y)\curl{}{\hat{u}}(y)\cdot\curl{}{\hat{\varphi}}(y) \,{\rm d}y.
\end{flalign*}
Hence, 
a function 
$u\in[H^1_{\#}(Q)]^3$ solves  
\begin{equation}
\label{cell1}
\int_Q A(y')\curl{y'}{u(\sigma^{-1}y')} \cdot \curl{y'}{\varphi(\sigma^{-1}y')} 
\,{\rm d}y' = \int_Qf(y') \cdot \curl{y'}{\varphi}(\sigma^{-1}y') \,{\rm d}y' \quad
\forall\varphi\in [H^1_{\#}(Q)]^3,
\end{equation}
if and only if $\hat{u}$
solves  
\begin{equation}
\label{cell2}
\int_Q A(y)\curl{}{\hat{u}(y)}\cdot \curl{}{\hat{\varphi} (y)} \,{\rm d}y = \int_Q \sigma^{-1}f(\sigma y) \cdot \curl{}{\hat{\varphi}}(y) \,{\rm d}y \quad \forall \hat{\varphi} \in[H^1_{\#}(Q)]^3.
\end{equation}

Let us now prove \eqref{ahomsym}. For fixed $\xi, \eta \in \RR^3$ let $N_\xi$ be the unique solution to \eqref{cellxi} and set $u(y) : = N_\xi( \sigma y)$, $y \in Q$. By \eqref{defahom},  assumption (\ref{asym}) and \eqref{inverse} we deduce that
\begin{flalign}
A^{\rm hom}\xi\cdot\eta &= \int_{Q}A(y')\bigl({\rm curl}_{y'}\,N_\xi (y') + \xi\bigr)\cdot\eta\,{\rm d}y'= \int_{Q}A(y')
\bigl({\rm curl}_{y'} u(\sigma^{-1}y') + \xi\bigr)\cdot\eta\,{\rm d}y' \nonumber \\
& \overset{\eqref{inverse}}{=} \int_{Q}A(y')
\bigl( \sigma {\rm curl} \hat{u}(\sigma^{-1}y') + \xi\bigr)\cdot\eta\,{\rm d}y' \overset{y'=\sigma y}{=}\int_{Q}A(\sigma y)\bigl(\sigma {\rm curl}\,\hat{u}(y) + \xi\bigr)\cdot\eta\,{\rm d}y \nonumber \\
& \overset{\eqref{asym}}{=} \int_{Q}A(y)
\bigl( {\rm curl} \hat{u}(y) + \sigma^{-1}\xi\bigr)\cdot\sigma^{-1}\eta\,{\rm d}y. \label{ahom2}
\end{flalign}
Since $N_\xi(y')$ solves \eqref{cellxi}, $u(\sigma^{-1}y')$ solves \eqref{cell1} for $f(y') = A(y') \xi$ and therefore $\hat{u}$ solves \eqref{cell2} where, by \eqref{asym}, $\sigma^{-1} f(\sigma y) = \sigma^{-1} A(\sigma y) \xi = A(y) \sigma^{-1} \xi$. Hence, the solution $N_{\sigma^{-1} \xi}$ to \eqref{cellxi}, for $\zeta = \sigma^{-1} \xi$,  is the projection of $\hat{u}$ onto the space $\{ u \in [H^1_{\#}(Q)]^3 : 
A\,{\rm curl}\,u =0 \}^\perp$ and the expression in (\ref{ahom2}) equals $A^{\rm hom}\sigma^{-1}\xi\cdot\sigma^{-1}\eta.$ The assertion \eqref{ahomsym} follows, in view of the arbitrary choice of $\xi,$ $\eta,$ and the equality $\sigma A^{\rm hom} \sigma^{-1} = \sigma^{-1} A^{\rm hom} \sigma $ which holds since $\sigma$ is unitary and $A^{\rm hom}$ is symmetric.
\end{proof}

\begin{corollary}
\label{symmetry}
If  \eqref{asym} holds for $\sigma=\sigma_k,$ where $\sigma_k$ is the rotation by $\pi$ around the $x_k$-axis,
then $A^{{\rm hom}}_{kl} = 0$, for all $l \neq k.$ 
\end{corollary}
\begin{proof}
Indeed, say for $k=1$ \eqref{ahomsym} takes the form
$$
\left( \begin{matrix}
A^{\text{hom}}_{11} & A^{\text{hom}}_{12} & A^{\text{hom}}_{13} \\[3pt]
A^{\text{hom}}_{21} & A^{\text{hom}}_{22} & A^{\text{hom}}_{23} \\[3pt]
A^{\text{hom}}_{31} & A^{\text{hom}}_{32} & A^{\text{hom}}_{33} 
\end{matrix} \right) = \left( \begin{matrix}
A^{\text{hom}}_{11} & -A^{\text{hom}}_{12} & -A^{\text{hom}}_{13} \\[3pt]
-A^{\text{hom}}_{21} & A^{\text{hom}}_{22} & A^{\text{hom}}_{23} \\[3pt]
-A^{\text{hom}}_{31} & A^{\text{hom}}_{32} & A^{\text{hom}}_{33} 
\end{matrix} \right),
$$ 
and hence $A^{\text{hom}}_{12}=A^{\text{hom}}_{21}=A^{\text{hom}}_{13}=A^{\text{hom}}_{31}=0.$ 
\end{proof}
Similarly, direct calculation proves the following statement.
\begin{corollary}
	\label{symmetry.2}
	If  \eqref{asym} holds for $\sigma=\sigma_k,$ where $\sigma_k$ is the rotation by $\pi/2$ around the $x_k$-axis,
	then $A^{{\rm hom}}_{kl} = 0$, for all $l \neq k$ and $A^{\rm hom}_{ii}=A^{\rm hom}_{jj}$, $i,j \neq k$. 
\end{corollary}
\begin{prop}
Suppose that the characteristic function $\chi_0$ of the set $Q_0$ satisfies (\ref{asym}) with $a=\chi_0I.$
Then for all $\omega^2\notin\{0\}\cup\{\alpha_k\}_{k=1}^\infty$ the matrix $\Gamma(\omega),$ defined by (\ref{gamma}), 
(\ref{Peqn})--(\ref{cond3}), satisfies the property
\[
\Gamma(\omega)=\sigma\Gamma(\omega)\sigma^{-1}=\sigma^{-1}\Gamma(\omega)\sigma.
\]
\end{prop}
\begin{proof}
We make use of the representation (\ref{me8}) for $\Gamma(\omega)$ and of the equations (\ref{nonlocal_eq}) for the 
functions $\phi^k.$ Multiplying (\ref{nonlocal_eq}) by $\psi \in [C^\infty_0(Q_0)]^3$ and integrating by parts yields
\begin{multline}
\int_{Q_0}\epsilon_0^{-1}(y){\rm curl}\,\phi^k(y)\cdot{\rm curl}\,\psi(y)\,dy=\alpha_k\int_{Q_0}\int_{Q_0}{ G}(y-x)\,{\rm div}\,\phi^k(x)\,
{\rm div}\,\psi(y)\,dxdy \\
 +\int_{Q_0}\phi^k(y)\cdot\psi(y)\,dy=0.
\label{Gamma_rot}
\end{multline}
We claim that the functions $\sigma\phi^k$ satisfy the identity (\ref{Gamma_rot}) with $Q_0$ replaced by $\sigma Q_0:=\{x\in Q:
\sigma^{-1}x\in Q_0\}.$ Indeed, treating each of the terms in (\ref{Gamma_rot}) separately, we obtain

 \[
 \int_{\sigma Q_0}\sigma\phi^k(y')\cdot\sigma\psi(y')\,dy'=\int_{\sigma Q_0}\phi^k(y')\cdot\psi(y')\,dy'\stackrel{y'=\sigma y}{=}
 \int_{Q_0}\phi^k(y)\cdot\psi(y)\,dy,
 \]
 \\
\[
\int_{\sigma Q_0}\ep_0^{-1}(y'){\rm curl}(\sigma\phi^k)(y')\cdot{\rm curl}(\sigma\psi)(y')\,dy'\stackrel{y'=\sigma y}{=}\int_{Q_0}\ep_0^{-1}(y){\rm curl}\,\phi^k(y)\cdot{\rm curl}\,\psi(y)\,dy,
\]
and
\begin{gather*}
\int_{\sigma Q_0}\int_{\sigma Q_0}{ G}(y'-x')\,{\rm div}(\sigma\phi^k)(x')\,{\rm div}(\sigma\psi)(y')\,dx'dy'\ \ \ \ \ \ \ \ \ \ \ \ \ \ \ \ \ \ \ \ \ \ \ \ \ \ \ \ \ \ \ \ \ \ \ \ \ \ \ \ \ \ \ \ \ \ \ \ \ \ \ \ \ \ \ \ 
\\
\stackrel{\substack{x'=\sigma x,\\ y'=\sigma y}}{=}\int_{Q_0}\int_{Q_0}{ G}\bigl(\sigma(y-x)\bigr)\,{\rm div}\,\phi^k(x)\,{\rm div}\,\psi(y)\,dxdy
\\
\ \ \ \ \ \ \ \ \ \ \ \ \ \ \ \ \ \ \ \ \ \ \ \ \ \ \ \ \ \ \ \ \ \ \ \ \ \ \ \ \ \ \ \ \ \ \ \ \ \ \ \ \ \ \ \ =\int_{Q_0}\int_{Q_0}{ G}\bigl(y-x)\,{\rm div}\,\phi^k(x)\,{\rm div}\,\psi(y)\,dxdy,
\end{gather*}
where the invariance of the Green function ${G}$, under the rotation $\sigma$, holds due to the assumption.   

Hence, $\phi^k$ are invariant under rotation by $\sigma$ and using the formula (\ref{me8}) yields
\[
\Gamma(\omega)= \omega^2\sigma\sigma^{-1} + \omega^4\sum_{k=1}^\infty \frac{\bigl( \int_{Q_0} \sigma\phi^k\bigr)\otimes\bigl( \int_{Q_0}\sigma\phi^k\bigr)}{\alpha_k - \omega^2}=\sigma\Gamma(\omega)\sigma^{-1},\ \ \ \ \ \ \ \ \ \omega^2\notin\{0\}\cup\{\alpha_k\}_{k=1}^\infty,
\]
as required.
\end{proof}
By analogy with Corollary \ref{symmetry}, Corollary \ref{symmetry.2} we obtain the following statement.
\begin{corollary}
If  \eqref{asym} holds for $\sigma=\sigma_k,$ where $\sigma_k$ is the rotation by $\pi$ around the $x_k$-axis, then
$\Gamma_{kl}(\omega)=0$ for all $l \neq k,$ $\omega^2\notin\{0\}\cup\{\alpha_k\}_{k=1}^\infty.$  Moreover, if $\sigma_k$ is a rotation by $\pi/2$ around the $x_k$-axis, then $\Gamma_{ii}(\omega)= \Gamma_{jj}(\omega)$ for $i,j \neq k$.
\end{corollary}







\section{Achknowledgements}
KDC and SC are grateful for the financial support of the Leverhulme Trust (Grant RPG--167 ``Dissipative and non-self-adjoint problems'') and the Engineering and Physical Sciences Research Council (Grant EP/L018802/2 ``Mathematical foundations of metamaterials: homogenisation, dissipation and operator theory'', and Grant EP/M017281/1 ``Operator asymptotics, a new approach to length-scale interactions in metamaterials.")
  
 We would also like to thank Valery Smyshlyaev for helpful discussions on the subject of this manuscript.



\end{document}